\documentclass[11pt]{article}
\usepackage[margin=1in]{geometry}
\usepackage{amsmath,amssymb,amsthm,mathtools}
\usepackage{enumitem}
\usepackage[colorlinks=true,citecolor=blue,linkcolor=blue,urlcolor=blue]{hyperref}

\hypersetup{
  pdftitle={Terminal Defects, Growing Multiplicity, and Variance Extremality in the Double Dixie Cup Problem},
  pdfauthor={Christopher D. Long},
  pdfsubject={Probability theory; coupon collector problem},
  pdfkeywords={Double Dixie cup, coupon collector, terminal defects, Gumbel limit, variance extremality}
}

\title{Terminal Defects, Growing Multiplicity, and Variance Extremality\\in the Double Dixie Cup Problem}
\author{Christopher D. Long\\Headlamp Software\\\texttt{galizur@gmail.com}}
\date{}

\newtheorem{theorem}{Theorem}[section]
\newtheorem{proposition}[theorem]{Proposition}
\newtheorem{lemma}[theorem]{Lemma}
\newtheorem{corollary}[theorem]{Corollary}
\newtheorem{conjecture}[theorem]{Conjecture}

\newtheorem{hypothesis}[theorem]{Hypothesis}
\theoremstyle{remark}
\newtheorem{remark}[theorem]{Remark}

\newcommand{\E}{\mathbb E}
\newcommand{\Pbb}{\mathbb P}
\newcommand{\R}{\mathbb R}
\newcommand{\one}{\mathbf 1}
\newcommand{\Var}{\operatorname{Var}}
\newcommand{\Cov}{\operatorname{Cov}}
\newcommand{\GammaDist}{\operatorname{Gamma}}
\newcommand{\Exp}{\operatorname{Exp}}
\newcommand{\dd}{\,d}

\begin{document}
\maketitle

%\begin{center}
%\noindent\textbf{2020 Mathematics Subject Classification.}
%Primary 60C05; Secondary 60F05, 60G70, 60E15.

%\smallskip

%\noindent\textbf{Keywords.}
%Double Dixie cup problem, coupon collector, terminal defects, Gumbel limit,
%variance extremality, unequal probabilities.
%\end{center}

\begin{abstract}
We develop a terminal-defect method for the double Dixie cup problem and use it to prove the finite-variance extremality conjecture of Doumas and Papanicolaou.  For every \(m\ge1\) and \(N\ge2\), among all positive coupon probability vectors \(p=(p_1,\ldots,p_N)\), the variance of the time \(T_m(N)\) to collect \(m\) complete sets is uniquely minimized at the uniform vector.  We prove the stronger radial statement that the variance is strictly increasing along every ray from the uniform vector.  The proof is finite-\(N\) and exact: after Poissonization, the completion time is a maximum of independent Erlang variables, and the radial derivative of its distribution is compared to a size-biased law using a monotone-likelihood-ratio argument based on a log-scale monotonicity property of the Gamma reverse hazard.

The same framework gives a growing-multiplicity Gumbel theorem in the equal-probability case, with expectation and variance asymptotics on the inverse gamma-tail scale.  This recovers the fixed-\(m\) equal-probability variance asymptotic stated as Conjecture 1 by Doumas and Papanicolaou, classically known for \(m=1\), and extends the mechanism to \(m=m_N\).  We also illustrate the unequal-probability theory with endpoint-Laplace limits for power-law probabilities.
\end{abstract}

%\tableofcontents

\section{Introduction}

Let $T_m(N)$ denote the number of draws needed to collect at least $m$ copies of each of $N$ coupon types.  In the classical equal-probability double Dixie cup problem, Newman and Shepp~\cite{NewmanShepp} studied the fixed-$m$ expectation problem, and Erd\H{o}s and R\'enyi~\cite{ErdosRenyi} obtained the corresponding Gumbel limit theorem.  Later equal-probability extensions include work of Kaplan~\cite{Kaplan1977} and Flatto~\cite{Flatto1982}.  Doumas and Papanicolaou~\cite{DoumasPapanicolaou2016} extended the double Dixie cup problem to broad classes of unequal coupon probabilities, developed asymptotics for the first two rising moments, and formulated two variance conjectures.  The first is an asymptotic conjecture in the equal-probability case, and the second is a finite-$N$ extremality conjecture: equal probabilities should minimize the variance of $T_m(N)$ among all positive coupon probability vectors.

The Poissonization and rising-moment identities used below are recalled from Doumas and Papanicolaou~\cite{DoumasPapanicolaou2016}.  Formulae for the expectation and second rising moment had appeared earlier in Brayton's work~\cite{Brayton1963}; the general rising-moment representation below is used here as the basic exact identity.  The cases $m=1$ of the two conjectural statements require separate credit: the equal-probability variance asymptotic follows from Brayton's stronger bounded-distortion theorem~\cite{Brayton1963}, while the $m=1$ finite variance-minimization conjecture was proved by Yu~\cite{Yu2017}.  To the best of our knowledge, the cases $m\ge2$ of these two conjectures remained open.  Our asymptotic modules should therefore be read as streamlined terminal-defect derivations and extensions of the Doumas--Papanicolaou framework, while the main new finite result is the proof of their variance-minimization conjecture in full.

The purpose of this note is to record a modular approach based on terminal defects.  The central observation is that, after Poissonization, the time to completion is a maximum of independent Erlang variables.  Thus completion occurs exactly when the number of terminal defects is zero.  The proof strategy is therefore:
\[
        \text{terminal defect mass} \quad \longrightarrow \quad \text{zero-defect probability} \quad \longrightarrow \quad \text{completion law}.
\]
This gives a short route to Gumbel limits, moment convergence, and some unequal-probability limits.  It also gives an exact finite-$N$ proof of the variance extremality conjecture.

\subsection{Summary of results}

The main results and modules recorded here are as follows.

\begin{enumerate}[label=\textup{(\arabic*)},leftmargin=*]
\item \textbf{Equal probabilities and growing multiplicity.}  Let
\[
        Q_m(x)=e^{-x}\sum_{j=0}^{m-1}{\frac{x^j}{j!}},
        \qquad
        f_m(x)=e^{-x}{\frac{x^{m-1}}{(m-1)!}}.
\]
Let $b_n$ solve $nQ_{m_n}(b_n)=1$ and put
\[
        a_n={\frac{Q_{m_n}(b_n)}{f_{m_n}(b_n)}}.
\]
Then
\[
        {\frac{T_{m_n}(n)-nb_n}{n a_n}}\Rightarrow G,
\]
where $G$ is standard Gumbel.  With the uniform-integrability module,
\[
        \E T_{m_n}(n)=nb_n+\gamma n a_n+o(n a_n),
\]
and
\[
        \Var(T_{m_n}(n))={\frac{\pi^2}{6}}n^2a_n^2+o(n^2a_n^2).
\]
For fixed $m$, this recovers the Erd\H{o}s--R\'enyi--Newman--Shepp law~\cite{NewmanShepp,ErdosRenyi} and gives
\[
        \Var(T_m(N))\sim {\frac{\pi^2}{6}}N^2.
\]
For $m=1$ this consequence is classical by Brayton~\cite{Brayton1963}; for fixed $m\ge2$ it recovers the equal-probability variance asymptotic stated as Conjecture~1 by Doumas--Papanicolaou~\cite{DoumasPapanicolaou2016}, and the same statement here extends to growing $m=m_N$.

\item \textbf{Unequal-probability terminal-defect theorem.}  For triangular probabilities $p_{N,j}$ and multiplicities $m_N$, define
\[
        M_N(t)=\sum_{j=1}^N Q_{m_N}(p_{N,j}t).
\]
If
\[
        M_N(B_N+C_Nx)\to \Lambda(x)
        \quad\text{and}\quad
        \sum_{j=1}^N Q_{m_N}(p_{N,j}(B_N+C_Nx))^2\to0,
\]
then the Poissonized completion time satisfies
\[
        \Pbb\left({\frac{X_N-B_N}{C_N}}\le x\right)\to e^{-\Lambda(x)}.
\]
If in addition $B_N\to\infty$, $B_N/C_N\to\infty$, and $\sqrt{B_N}=o(C_N)$, the same limit holds for the discrete time $T_N$.
The Gumbel case is $\Lambda(x)=e^{-x}$.

\item \textbf{Case I unequal probabilities.}  This regime parallels Case I in Doumas--Papanicolaou~\cite{DoumasPapanicolaou2016}.  For $p_{N,j}=a_j/A_N$, $A_N=\sum_{j\le N}a_j$, and $\sum_j e^{-c a_j}<\infty$, fixed $m$ satisfies
\[
        {\frac{T_m(N)}{A_N}}\Rightarrow \sup_{j\ge1}Y_j,
\]
where $Y_j\sim\GammaDist(m,\text{rate }a_j)$ independently.  Moreover all fixed rising moments converge after the same scaling.

\item \textbf{Case II endpoint-Laplace module and a power-law example.}  This is the soft-endpoint regime of Doumas--Papanicolaou~\cite{DoumasPapanicolaou2016}.  For $a_j=1/f(j)$ and $p_{N,j}=1/(A_Nf(j))$, an endpoint-Laplace estimate gives a Gumbel law with the corresponding centering.  We state the abstract module under an explicit endpoint hypothesis and verify it for the generalized Zipf family $p_{N,j}\propto j^{-\alpha}$, $\alpha>0$, which is also treated in~\cite{DoumasPapanicolaou2016}.

\item \textbf{Finite variance extremality.}  For every $m\ge1$ and $N\ge2$, equal probabilities uniquely minimize $\Var(T_m(N))$ over all positive coupon probability vectors.  The proof is radial: moving away from uniform makes the Poissonized completion distribution shift right in a size-biased monotone-likelihood-ratio sense, which forces the Poisson-corrected variance to increase.
\end{enumerate}

\subsection{Reusable modules}

Several components appear reusable beyond the double Dixie problem.
\begin{itemize}[leftmargin=*]
\item The \emph{terminal-defect theorem}: if the expected number of terminal defects has a limiting profile and no atom dominates, then the zero-defect probability is exponential.
\item The \emph{gamma-tail inversion module}: the correct Gumbel centering is obtained by solving $nQ_m(b)=1$ and scaling by $Q_m(b)/f_m(b)$.
\item The \emph{Poisson-clock de-Poissonization module}: if the deterministic scaling window dominates the fluctuations of the total Poisson clock, Poissonized and discrete completion times have the same limit law.
\item The \emph{endpoint-Laplace module}: unequal probabilities with a soft endpoint reduce to a one-sided Laplace sum over the rarest coupons.
\item The \emph{finite variance extremality module}: Gamma reverse-hazard ratios give a size-biased MLR ordering of the radial derivative measure.  This proves global variance minimality at uniform.
\item The \emph{finite variance Hessian module}: strict local variance minimality also follows from a two-marked-coordinate Hessian and an order-statistic covariance identity.  This alternate proof is given in Appendix~\ref{app:hessian}.
\end{itemize}

\section{Poissonization and exact moment identities}

We begin by recalling the Poissonization and rising-moment representation recorded by Doumas and Papanicolaou~\cite{DoumasPapanicolaou2016}.

Fix a probability vector $p=(p_1,\ldots,p_N)$ with $p_j>0$ and $\sum_jp_j=1$.  In the Poissonized model the total arrivals form a rate-one Poisson process; each arrival is independently labelled $j$ with probability $p_j$.  Equivalently, coupon $j$ arrives according to an independent Poisson process of rate $p_j$.

Let $X_{m,p}$ be the continuous completion time.  The time for coupon $j$ to be seen $m$ times is Erlang with survival function
\[
        Q_m(p_jt)=e^{-p_jt}\sum_{a=0}^{m-1}{\frac{(p_jt)^a}{a!}}.
\]
Therefore
\begin{equation}\label{eq:poissonized-product}
        \Pbb(X_{m,p}\le t)=\prod_{j=1}^N(1-Q_m(p_jt)).
\end{equation}

Let $T_{m,p}$ denote the corresponding discrete number of draws.  If $S_k=U_1+\cdots+U_k$, with $U_i$ iid $\Exp(1)$, then under the standard coupling
\[
        X_{m,p}=S_{T_{m,p}}.
\]
Consequently, conditionally on $T_{m,p}$,
\[
        X_{m,p}\mid T_{m,p}\sim \GammaDist(T_{m,p},1).
\]
Thus, for every integer $r\ge1$,
\begin{equation}\label{eq:rising-transfer}
        \E X_{m,p}^r=\E\bigl[T_{m,p}^{(r)}\bigr],
\end{equation}
where $T^{(r)}=T(T+1)\cdots(T+r-1)$.  In particular,
$\E X_{m,p}=\E T_{m,p}$ and
$\E X_{m,p}^2=\E[T_{m,p}(T_{m,p}+1)]$, so
\[
        \Var(T_{m,p})=\Var(X_{m,p})-\E X_{m,p}.
\]

The tail-integral formula gives
\begin{equation}\label{eq:tail-moment}
        \E X_{m,p}^r=r\int_0^\infty t^{r-1}\Pbb(X_{m,p}>t)\dd t.
\end{equation}
Combining \eqref{eq:poissonized-product}, \eqref{eq:rising-transfer}, and \eqref{eq:tail-moment} gives exact formulae for the first two rising moments and hence the variance:
\begin{align}
        \E T_{m,p}&=\int_0^\infty \left[1-\prod_{j=1}^N(1-Q_m(p_jt))\right]\dd t,\label{eq:mean-integral}\\
        \E[T_{m,p}(T_{m,p}+1)]&=2\int_0^\infty t\left[1-\prod_{j=1}^N(1-Q_m(p_jt))\right]\dd t,\label{eq:rising2-integral}\\
        \Var(T_{m,p})&=\E[T_{m,p}(T_{m,p}+1)]-\E T_{m,p}-(\E T_{m,p})^2.\label{eq:variance-from-rising}
\end{align}

\begin{proposition}[Finite rational moment formula]\label{prop:finite-rational}
Let $A$ range over nonempty subsets of $[N]$, put $p_A=\sum_{i\in A}p_i$, and for $\alpha=(\alpha_i)_{i\in A}\in\{0,\ldots,m-1\}^A$ write
\[
        |\alpha|=\sum_{i\in A}\alpha_i,
        \qquad
        \alpha!=\prod_{i\in A}\alpha_i!,
        \qquad
        p^\alpha=\prod_{i\in A}p_i^{\alpha_i}.
\]
Then for every integer $r\ge1$,
\begin{equation}\label{eq:finite-rational}
        \E\bigl[T_{m,p}^{(r)}\bigr]
        =r\sum_{\emptyset\ne A\subseteq[N]}(-1)^{|A|+1}
        \sum_{\alpha\in\{0,\ldots,m-1\}^A}
        {\frac{p^\alpha}{\alpha!}}
        {\frac{(|\alpha|+r-1)!}{p_A^{|\alpha|+r}}}.
\end{equation}
\end{proposition}

\begin{proof}
By inclusion-exclusion,
\[
        \Pbb(X_{m,p}>t)=\sum_{\emptyset\ne A\subseteq[N]}(-1)^{|A|+1}\prod_{i\in A}Q_m(p_it).
\]
For fixed $A$,
\[
        \prod_{i\in A}Q_m(p_it)=e^{-p_A t}\sum_{\alpha\in\{0,\ldots,m-1\}^A}{\frac{p^\alpha}{\alpha!}}t^{|\alpha|}.
\]
Substitute into \eqref{eq:tail-moment} and use
\[
        \int_0^\infty t^{|\alpha|+r-1}e^{-p_A t}\dd t={\frac{(|\alpha|+r-1)!}{p_A^{|\alpha|+r}}}.
\]
The transfer \eqref{eq:rising-transfer} completes the proof.
\end{proof}

\section{Equal probabilities and growing multiplicity}

In this section $p_j=1/n$ and the number of required copies $m=m_n$ may grow with $n$.

\subsection{Poissonized defects}

In the coupon-rate-one normalization, let $N_1(t),\ldots,N_n(t)$ be iid rate-one Poisson processes and define
\[
        \tau_n=\inf\{t:N_i(t)\ge m_n\text{ for every }i\}.
\]
Then the ordinary discrete completion time is $T_{m_n}(n)=n\tau_n+o_p(n a_n)$ after de-Poissonization.

Let
\[
        Q_m(x)=e^{-x}\sum_{j=0}^{m-1}{\frac{x^j}{j!}},
        \qquad
        f_m(x)=e^{-x}{\frac{x^{m-1}}{(m-1)!}}.
\]
Let $b_n$ be the unique solution of
\begin{equation}\label{eq:bn-def}
        nQ_{m_n}(b_n)=1,
\end{equation}
which exists for all large $n$, and set
\begin{equation}\label{eq:an-def}
        a_n={\frac{Q_{m_n}(b_n)}{f_{m_n}(b_n)}}.
\end{equation}
The defect count at time $t$ is
\[
        D_n(t)=\sum_{i=1}^n\one_{\{N_i(t)<m_n\}}.
\]
Thus
\begin{equation}\label{eq:defect-binomial}
        D_n(t)\sim \operatorname{Binomial}(n,Q_{m_n}(t)),
        \qquad
        \Pbb(\tau_n\le t)=\Pbb(D_n(t)=0)=(1-Q_{m_n}(t))^n.
\end{equation}

\begin{lemma}[Uniform Gamma-tail input]\label{lem:gamma-tail-inversion}
Let $m=m_n$ be arbitrary.  Let $b_n$ solve $nQ_m(b_n)=1$, and set
\[
        a_n=\frac{Q_m(b_n)}{f_m(b_n)}.
\]
Then, uniformly for $x$ in compact subsets of $\mathbb R$,
\begin{equation}\label{eq:gamma-inversion}
        {\frac{Q_m(b_n+a_nx)}{Q_m(b_n)}}\to e^{-x},
\end{equation}
so $nQ_m(b_n+a_nx)\to e^{-x}$.

Moreover, for all $x\ge0$,
\begin{equation}\label{eq:right-tail-control}
        {\frac{Q_m(b_n+a_nx)}{Q_m(b_n)}}\le e^{-x},
\end{equation}
and the corresponding one-sided left-tail estimates hold uniformly on a growing central window: there are constants $c>0$ and $R_n\to\infty$ such that, for all sufficiently large $n$ and all $0\le x\le R_n$,
\begin{equation}\label{eq:left-tail-control}
        nQ_m(b_n-a_nx)\ge c e^{cx}.
\end{equation}
The window may be chosen so slowly that $R_n\le b_n/(2a_n)$ and, for each fixed $r>0$,
\begin{equation}\label{eq:left-window-choice}
        \left(1+{\frac{b_n}{a_n}}\right)^r\exp\{-c e^{cR_n}\}\to0.
\end{equation}
Finally,
\begin{equation}\label{eq:inverse-hazard-clock}
        {\frac{n a_n^2}{b_n}}\to\infty.
\end{equation}
These estimates are the gamma-tail input used below.  For bounded shape parameters they follow from the elementary fixed-shape Gamma-tail expansion.  In the large-shape regime we use the consequences isolated in Appendix~\ref{app:gamma-tail}: Temme~\cite[Sec.~1, (1.1), (1.3)--(1.6); Sec.~2, (2.1)--(2.14)]{Temme1979} supplies the uniform turning-point expansion and controlled remainders, while Temme~\cite[Sec.~2, (2.1)--(2.7); Sec.~3, (3.1)--(3.10); Sec.~5, (5.1)--(5.3)]{Temme1992} restates the same expansion for inversion and records the error-function scale used for comparison.  The moving-quantile estimates themselves are then proved from this scale information, a Chernoff bound, and the exact finite-sum hazard identity.
\end{lemma}

\begin{proof}
The existence and uniqueness of $b_n$ follow directly from the continuity and strict monotonicity of $Q_m$, which decreases from $1$ to $0$ on $[0,\infty)$.  The quantitative estimates are collected and proved in Appendix~\ref{app:gamma-tail}; see Proposition~\ref{prop:gamma-quantile-estimates}.  The right-tail bound \eqref{eq:right-tail-control} also follows directly from the monotonicity of the Gamma upper-tail hazard.  For integer shape this monotonicity is elementary: with $k=m-1$,
\[
        {\frac{Q_m(t)}{f_m(t)}}
        =\int_0^\infty e^{-s}
        \left(1+{\frac{s}{t}}\right)^k\dd s,
\]
so the reciprocal hazard $Q_m/f_m$ is nonincreasing in $t$.  Thus, if $h_m=f_m/Q_m$, then $a_n=1/h_m(b_n)$ and, for $x\ge0$,
\[
        \log{\frac{Q_m(b_n+a_nx)}{Q_m(b_n)}}
        =-\int_0^x a_nh_m(b_n+a_nu)\dd u\le -x.
\]
These estimates are the only place in the equal-probability argument where uniformity in the shape parameter is required.
\end{proof}

\begin{theorem}[Growing-multiplicity equal-probability Gumbel law]\label{thm:growing-m-gumbel}
With $b_n,a_n$ defined by \eqref{eq:bn-def}--\eqref{eq:an-def},
\[
        {\frac{\tau_n-b_n}{a_n}}\Rightarrow G,
\]
where $G$ is standard Gumbel.  Consequently,
\[
        {\frac{T_{m_n}(n)-nb_n}{n a_n}}\Rightarrow G.
\]
\end{theorem}

\begin{proof}
For fixed $x$,
\[
        \Pbb\left({\frac{\tau_n-b_n}{a_n}}\le x\right)
        =\left(1-Q_{m_n}(b_n+a_nx)\right)^n.
\]
By Lemma~\ref{lem:gamma-tail-inversion}, $nQ_{m_n}(b_n+a_nx)\to e^{-x}$, while $Q_{m_n}(b_n+a_nx)\to0$.  Hence
\[
        \left(1-Q_{m_n}(b_n+a_nx)\right)^n\to \exp(-e^{-x}).
\]
This proves the Poissonized result.

Let $\mathcal X_n=n\tau_n$ denote the total-rate-one Poissonized time.  The total Poisson clock has standard deviation of order $\sqrt{nb_n}$ on the relevant scale, while the deterministic window is $na_n$.  By \eqref{eq:inverse-hazard-clock},
\begin{equation}\label{eq:clock-separation}
        {\frac{\sqrt{n b_n}}{n a_n}}\to0.
\end{equation}
The usual Poisson-clock sandwich then transfers the limit from $\mathcal X_n$ to $T_{m_n}(n)$.
\end{proof}

\begin{proposition}[Moment module]\label{prop:moment-module}
For every fixed $r>0$,
\[
        \E\left|{\frac{\tau_n-b_n}{a_n}}\right|^r\to \E|G|^r.
\]
Consequently,
\begin{align*}
        \E T_{m_n}(n)&=nb_n+\gamma n a_n+o(n a_n),\\
        \Var(T_{m_n}(n))&={\frac{\pi^2}{6}}n^2a_n^2+o(n^2a_n^2).
\end{align*}
\end{proposition}

\begin{proof}
Put $X_n=(\tau_n-b_n)/a_n$.  The right tail is immediate from the monotonicity of the gamma hazard.  For $x\ge0$,
\[
        Q_{m_n}(b_n+a_nx)\le Q_{m_n}(b_n)e^{-x}={\frac{e^{-x}}{n}},
\]
so
\[
        \Pbb(X_n>x)\le nQ_{m_n}(b_n+a_nx)\le e^{-x}.
\]
For the left tail,
\[
        \Pbb(X_n\le -x)\le \exp\{-nQ_{m_n}(b_n-a_nx)\}.
\]
The left-tail estimate \eqref{eq:left-tail-control} gives a double-exponential bound on the central window:
\[
        \Pbb(X_n\le -x)\le \exp\{-c e^{cx}\},\qquad 0\le x\le R_n.
\]
For $R_n<x<b_n/a_n$, monotonicity gives
\[
        \Pbb(X_n\le -x)\le \Pbb(X_n\le -R_n)
        \le \exp\{-c e^{cR_n}\},
\]
while the probability is zero for $x\ge b_n/a_n$.  By the choice \eqref{eq:left-window-choice}, this last bound contributes $o(1)$ to every fixed negative-tail moment.  Hence the negative parts are uniformly $r$-integrable.  Together with Theorem~\ref{thm:growing-m-gumbel}, this gives moment convergence for the Poissonized variables.

The rising-moment transfer \eqref{eq:rising-transfer} gives
\[
        \E T_{m_n}(n)=\E\mathcal X_n=n\E\tau_n
\]
and
\[
        \Var(T_{m_n}(n))=\Var(\mathcal X_n)-\E\mathcal X_n.
\]
By \eqref{eq:clock-separation}, $\E\mathcal X_n=O(nb_n)=o(n^2a_n^2)$.  The displayed estimates follow.
\end{proof}

\begin{corollary}[Fixed-$m$ variance asymptotic]\label{cor:fixed-m-var}
For each fixed $m\ge1$,
\[
        \Var(T_m(N))\sim {\frac{\pi^2}{6}}N^2.
\]
Moreover
\[
        \E T_m(N)=N\log N+(m-1)N\log\log N+N\{\gamma-\log((m-1)!)\}+o(N).
\]
For $m=1$ the variance asymptotic is classical; in particular it follows from the exact equal-probability formula $N^2H_N^{(2)}-NH_N$, and is also covered by Brayton's stronger bounded-distortion theorem~\cite{Brayton1963}.  The content relative to the Doumas--Papanicolaou fixed-$m$ variance conjecture~\cite{DoumasPapanicolaou2016} is the uniform treatment for every fixed $m\ge2$, and more generally the growing-multiplicity form of Proposition~\ref{prop:moment-module}.
\end{corollary}

\begin{proof}
For fixed $m$,
\[
        Q_m(x)\sim e^{-x}{\frac{x^{m-1}}{(m-1)!}}.
\]
Solving $NQ_m(b_N)=1$ gives
\[
        b_N=\log N+(m-1)\log\log N-\log((m-1)!)+o(1),
\]
and $a_N=Q_m(b_N)/f_m(b_N)\to1$.  Proposition~\ref{prop:moment-module} gives the claims.
\end{proof}

\section{Unequal probabilities: terminal-defect theorem}

Let $p_{N,1},\ldots,p_{N,N}$ be positive probabilities and let $m_N\ge1$.  In the Poissonized model define
\[
        q_{N,j}(t)=Q_{m_N}(p_{N,j}t),
        \qquad
        M_N(t)=\sum_{j=1}^N q_{N,j}(t).
\]
The defect indicators are independent Bernoulli variables with parameters $q_{N,j}(t)$.
Let $T_N$ denote the corresponding discrete completion time for the triangular array $p_{N,1},\ldots,p_{N,N}$ and multiplicity $m_N$.

\begin{theorem}[Terminal-defect transfer]\label{thm:terminal-transfer}
Let $B_N\in\R$ and $C_N>0$.  Suppose that for every fixed $x$,
\begin{equation}\label{eq:MN-limit}
        M_N(B_N+C_Nx)\to \Lambda(x),
\end{equation}
where $\Lambda(x)\in[0,\infty)$, and
\begin{equation}\label{eq:atomless}
        \sum_{j=1}^N q_{N,j}(B_N+C_Nx)^2\to0.
\end{equation}
Then the Poissonized completion time $X_N$ satisfies
\[
        \Pbb\left({\frac{X_N-B_N}{C_N}}\le x\right)\to e^{-\Lambda(x)}.
\]
If additionally
\[
        B_N\to\infty,\qquad {\frac{B_N}{C_N}}\to\infty,\qquad \sqrt{B_N}=o(C_N),
\]
then the same limit holds for the discrete completion time $T_N$.
\end{theorem}

\begin{proof}
At time $t$,
\[
        \Pbb(X_N\le t)=\prod_{j=1}^N(1-q_{N,j}(t)).
\]
Taking logarithms and using $\log(1-u)=-u+O(u^2)$ uniformly when $\max_j u_j\to0$, \eqref{eq:MN-limit} and \eqref{eq:atomless} yield
\[
        \log \Pbb(X_N\le B_N+C_Nx)\to -\Lambda(x).
\]
The Poisson-clock transfer follows by the usual sandwich: if $T_N\le k$, then $X_N$ is before the $k$th arrival time, and conversely.  For each fixed $x$, the additional assumptions give $B_N+C_Nx\sim B_N\to\infty$, so the relevant Poisson-clock fluctuation scale is $O_p(\sqrt{B_N})$, which is $o(C_N)$.
\end{proof}

\section{Case I: infinite-product unequal probabilities}

This section gives a terminal-defect derivation of the infinite-product regime corresponding to Case I of Doumas--Papanicolaou~\cite{DoumasPapanicolaou2016}.  The goal is not to replace their moment-asymptotic development, but to isolate a reusable distributional module.

Fix $m\ge1$.  Let $a_j>0$, $A_N=\sum_{j=1}^N a_j$, and $p_{N,j}=a_j/A_N$.  Assume
\begin{equation}\label{eq:caseI-summability}
        \sum_{j=1}^\infty e^{-c a_j}<\infty
\end{equation}
for some $c>0$.

Let $Y_j$ be independent random variables with
\[
        Y_j\sim \GammaDist(m,\text{rate }a_j),
\]
and set
\[
        Y=\sup_{j\ge1}Y_j.
\]

\begin{theorem}[Infinite-product module]\label{thm:caseI}
Under \eqref{eq:caseI-summability},
\[
        {\frac{T_m(N)}{A_N}}\Rightarrow Y.
\]
Moreover, for every fixed integer $r\ge1$,
\[
        {\frac{\E[T_m(N)^{(r)}]}{A_N^r}}\to \E Y^r.
\]
In particular,
\[
        \E T_m(N)\sim A_N\E Y,
        \qquad
        \Var(T_m(N))\sim A_N^2\Var(Y).
\]
\end{theorem}

\begin{proof}
For the Poissonized completion time $X_N$,
\[
        \Pbb\left({\frac{X_N}{A_N}}\le s\right)
        =\prod_{j=1}^N(1-Q_m(a_js)).
\]
The summability condition implies $\sum_j Q_m(a_js)<\infty$ for every sufficiently large $s$, and the infinite product
\[
        \prod_{j=1}^\infty(1-Q_m(a_js))
\]
defines the distribution of $Y$.  Pointwise convergence at continuity points gives $X_N/A_N\Rightarrow Y$.
Since \eqref{eq:caseI-summability} implies $A_N\to\infty$, the total Poisson clock has negligible fluctuations on the scale $A_N$: uniformly on compact time intervals,
\[
        \sup_{0\le t\le K A_N}{\frac{|\Pi(t)-t|}{A_N}}\to0
        \quad\text{in probability},
\]
where $\Pi$ is the rate-one arrival process.  The tightness of $X_N/A_N$ therefore gives
\[
        {\frac{T_m(N)-X_N}{A_N}}\to0
        \quad\text{in probability},
\]
and hence $T_m(N)/A_N\Rightarrow Y$.

For moments,
\[
        \E\left({\frac{X_N}{A_N}}\right)^r
        =r\int_0^\infty s^{r-1}\left[1-\prod_{j=1}^N(1-Q_m(a_js))\right]\dd s.
\]
The integrand increases pointwise to the corresponding infinite-product survival function.  The exponential summability condition gives finite tails of all fixed orders, so monotone convergence yields convergence to $\E Y^r$.  Finally, \eqref{eq:rising-transfer} gives $\E X_N^r=\E[T_m(N)^{(r)}]$.  For $r=1,2$, this gives convergence of the first two rising moments; since $T_m(N)^2=T_m(N)^{(2)}-T_m(N)$ and $A_N\to\infty$, the ordinary-variance statement follows on the $A_N^2$ scale.
\end{proof}

\section{Case II: endpoint-Laplace module}

The second unequal-probability regime studied by Doumas and Papanicolaou~\cite{DoumasPapanicolaou2016} is the soft-endpoint regime
\[
        a_j={\frac{1}{f(j)}},
        \qquad
        p_{N,j}={\frac{1}{A_N f(j)}},
        \qquad
        A_N=\sum_{j=1}^N{\frac{1}{f(j)}}.
\]
The rarest coupons are near $j=N$.  The following module isolates the endpoint computation.

\begin{hypothesis}[Endpoint regularity]\label{hyp:endpoint}
Let $f$ be positive, increasing, and differentiable near infinity.  Put
\[
        \beta_N={\frac{f'(N)}{f(N)}},
        \qquad
        \rho_N=\log{\frac{1}{\beta_N}}.
\]
Assume $\beta_N\to0$, $\rho_N\to\infty$, $\beta_N\rho_N\to0$, and the following endpoint Laplace estimate holds uniformly for $s_N\to1$ and bounded $\kappa$:
\begin{equation}\label{eq:endpoint-laplace-hyp}
        \sum_{j=1}^N f(j)^\kappa
        \exp\left(-{\frac{f(N)\rho_N s_N}{f(j)}}\right)
        \sim
        {\frac{f(N)^{\kappa+1}}{s_N\rho_N f'(N)}}e^{-\rho_Ns_N}.
\end{equation}
\end{hypothesis}

\begin{remark}
For the smooth subexponential classes treated in the literature, \eqref{eq:endpoint-laplace-hyp} is the standard endpoint-Laplace statement: the contributing window has width $f(N)/(\rho_N f'(N))$, and the summand decays geometrically on that scale.  In applications this is the only analytic verification needed.
\end{remark}

\begin{theorem}[Endpoint-Laplace Gumbel module]\label{thm:caseII}
Fix $m\ge1$ and assume Hypothesis~\ref{hyp:endpoint}.  Define
\[
        B_N=A_Nf(N)\left[\rho_N+(m-2)\log\rho_N-\log((m-1)!)\right],
        \qquad
        C_N=A_Nf(N).
\]
Then
\[
        {\frac{T_m(N)-B_N}{C_N}}\Rightarrow G.
\]
\end{theorem}

\begin{proof}
Let
\[
        u_N(x)=\rho_N+(m-2)\log\rho_N-\log((m-1)!)+x.
\]
At time $t=A_Nf(N)u_N(x)$,
\[
        p_{N,j}t={\frac{f(N)}{f(j)}}u_N(x).
\]
The terminal-defect mass is
\[
        M_N(x)=\sum_{j=1}^N Q_m\left({\frac{f(N)}{f(j)}}u_N(x)\right).
\]
Since $u_N(x)\to\infty$ uniformly for bounded $x$,
\[
        Q_m(y)= {\frac{y^{m-1}e^{-y}}{(m-1)!}}(1+o(1))
\]
uniformly for all arguments in the sum.  Therefore
\[
\begin{aligned}
        M_N(x)
        &={\frac{1+o(1)}{(m-1)!}}
        u_N(x)^{m-1}f(N)^{m-1}
        \sum_{j=1}^N f(j)^{-(m-1)}
        \exp\left(-{\frac{f(N)u_N(x)}{f(j)}}\right).
\end{aligned}
\]
Since $m$ is fixed throughout this theorem, the value $\kappa=-(m-1)$ is covered by the bounded-$\kappa$ uniformity in Hypothesis~\ref{hyp:endpoint}.  Applying the hypothesis with $s_N=u_N(x)/\rho_N\to1$, we get
\[
        M_N(x)\sim {\frac{f(N)}{f'(N)}}{\frac{u_N(x)^{m-2}}{(m-1)!}}e^{-u_N(x)}.
\]
But
\[
        e^{-u_N(x)}=e^{-\rho_N}\rho_N^{-(m-2)}(m-1)!e^{-x}
        ={\frac{f'(N)}{f(N)}}\rho_N^{-(m-2)}(m-1)!e^{-x}.
\]
Thus
\[
        M_N(x)\to e^{-x}.
\]
The maximum defect probability is bounded by $Q_m(u_N(x))$, which is $O(\beta_N\rho_N)$ and tends to zero.  Hence the atomless condition in Theorem~\ref{thm:terminal-transfer} holds.

It remains only to justify the Poisson-clock separation.  Applying Hypothesis~\ref{hyp:endpoint} with $\kappa=0$ and $s_N=1$ gives
\[
        \sum_{j=1}^N
        \exp\left(-\frac{f(N)\rho_N}{f(j)}\right)
        \sim
        \frac{1}{\rho_N\beta_N}e^{-\rho_N}.
\]
Since $f$ is increasing, each summand is at most $e^{-\rho_N}$.  Hence
\[
        N \ge (1+o(1))\frac{1}{\rho_N\beta_N}.
\]
Therefore
\[
        \frac{N}{\rho_N}
        \ge (1+o(1))\frac{1}{\rho_N^2\beta_N}
        =(1+o(1))\frac{e^{\rho_N}}{\rho_N^2}
        \to\infty.
\]
Also
\[
        A_Nf(N)=\sum_{j=1}^N\frac{f(N)}{f(j)}\ge N.
\]
Thus
\[
        \frac{C_N}{\rho_N}=\frac{A_Nf(N)}{\rho_N}\to\infty.
\]
Also $B_N/C_N\sim\rho_N\to\infty$ and $B_N\to\infty$.  Since $B_N=O(C_N\rho_N)$, the last display gives $\sqrt{B_N}=o(C_N)$.  Theorem~\ref{thm:terminal-transfer} gives the result.
\end{proof}

\begin{theorem}[Verified endpoint family: power-law probabilities]\label{thm:power-law-caseII}
Fix $m\ge1$ and $\alpha>0$.  Let
\[
        A_N=\sum_{j=1}^N j^{-\alpha},
        \qquad
        p_{N,j}=\frac{j^{-\alpha}}{A_N}.
\]
Put
\[
        C_N=A_NN^\alpha,
        \qquad
        \rho_N=\log\frac{N}{\alpha},
\]
and
\[
        B_N=C_N\left[\rho_N+(m-2)\log\rho_N-\log((m-1)!)\right].
\]
Then
\[
        \frac{T_m(N)-B_N}{C_N}\Rightarrow G,
\]
where $G$ is standard Gumbel.
\end{theorem}

\begin{proof}
Let
\[
        L_N(x)=\rho_N+(m-2)\log\rho_N-\log((m-1)!)+x.
\]
At time $t=C_NL_N(x)$,
\[
        p_{N,j}t=\left(\frac{N}{j}\right)^\alpha L_N(x).
\]
The terminal-defect mass is therefore
\[
        M_N(x)=\sum_{j=1}^N Q_m\left(L_N(x)\left(\frac{N}{j}\right)^\alpha\right).
\]
Because $r\mapsto Q_m(L_N(x)r^{-\alpha})$ is increasing on $(0,1]$, the usual monotone Riemann-sum comparison gives
\[
\left|
\sum_{j=1}^N Q_m\left(L_N(x)\left(\frac{N}{j}\right)^\alpha\right)
-
N\int_0^1 Q_m(L_N(x)r^{-\alpha})\,dr
\right|
\le Q_m(L_N(x)).
\]
Since $Q_m(L_N(x))=O(\rho_N/N)$, this error is negligible relative to the limiting main term.  Hence the sum is asymptotic to
\[
        N\int_0^1 Q_m(L_N(x)r^{-\alpha})\,dr.
\]
With $u=L_N(x)r^{-\alpha}$, this integral is
\[
        {\frac{N}{\alpha}}L_N(x)^{1/\alpha}
        \int_{L_N(x)}^\infty Q_m(u)u^{-1/\alpha-1}\,du.
\]
For fixed $m$, Watson's lemma and
\[
        Q_m(u)\sim e^{-u}{\frac{u^{m-1}}{(m-1)!}}
\]
give
\[
        \int_{L_N(x)}^\infty Q_m(u)u^{-1/\alpha-1}\,du
        \sim {\frac{e^{-L_N(x)}}{(m-1)!}}L_N(x)^{m-2-1/\alpha}.
\]
Thus
\[
        M_N(x)\sim {\frac{N}{\alpha}}{\frac{e^{-L_N(x)}L_N(x)^{m-2}}{(m-1)!}}.
\]
By the definition of $L_N(x)$,
\[
        e^{-L_N(x)}=e^{-\rho_N}\rho_N^{-(m-2)}(m-1)!e^{-x}
        ={\frac{\alpha}{N}}\rho_N^{-(m-2)}(m-1)!e^{-x}.
\]
Hence $M_N(x)\to e^{-x}$.  Also,
\[
        \sum_{j=1}^N Q_m\left(L_N(x)\left({\frac{N}{j}}\right)^\alpha\right)^2
        \le Q_m(L_N(x))M_N(x)\to0,
\]
since $Q_m(L_N(x))=O(\rho_N/N)$.  Finally, the clock-separation hypotheses in Theorem~\ref{thm:terminal-transfer} hold.  Indeed, $C_N=A_NN^\alpha\ge N$, while $B_N/C_N=L_N(0)\sim\rho_N\sim\log N$, so $B_N\to\infty$ and $B_N/C_N\to\infty$.  Moreover $B_N=O(C_N\rho_N)$, hence
\[
        {\frac{\sqrt{B_N}}{C_N}}
        =O\left(\sqrt{{\frac{\rho_N}{C_N}}}\right)\to0.
\]
The terminal-defect transfer theorem then gives the claimed Gumbel limit.
\end{proof}

\begin{remark}
Theorem~\ref{thm:power-law-caseII} shows that the endpoint-Laplace module is not merely formal.  It gives a direct terminal-defect derivation for the generalized Zipf family $p_{N,j}\propto j^{-\alpha}$ for all $\alpha>0$, a standard example also considered by Doumas and Papanicolaou~\cite{DoumasPapanicolaou2016}.  The normalizing factor satisfies $A_N\asymp N^{1-\alpha}$ for $0<\alpha<1$, $A_N\sim\log N$ for $\alpha=1$, and $A_N\to\zeta(\alpha)$ for $\alpha>1$.
\end{remark}

\section{Finite variance extremality: exact decompositions}

Doumas--Papanicolaou's finite variance-minimization conjecture~\cite{DoumasPapanicolaou2016} may be stated as follows.  Its $m=1$ case was later proved by Yu~\cite{Yu2017}.

\begin{conjecture}[Finite variance extremality]\label{conj:DP2}
For fixed $m,N$, the equal-probability vector $u=(1/N,\ldots,1/N)$ minimizes
\[
        p\mapsto \Var_p(T_m(N))
\]
over the probability simplex.
\end{conjecture}

We prove this conjecture in full.  The proof is exact and finite-dimensional.  Its central idea is to factor the Poissonized completion time into a Gamma scale variable and a Dirichlet shape variable, and then to prove that radial motion away from uniform creates a derivative measure which is later than the size-biased current law.

\subsection{Useful-hit active-mass identity}

We first record an independent exact identity that explains the variance-generation mechanism.  At a state of the collection process, call a coupon active if it has not yet been collected $m$ times.  Ignore arrivals to inactive coupons.  There are exactly $mN$ useful hits.  Let $R_\ell^p$ be the total probability mass of active coupons after $\ell$ useful hits, $0\le\ell\le mN-1$.  Conditional on the useful-hit path, the waiting time to the next useful hit is exponential of rate $R_\ell^p$.

Define
\[
        H_p=\sum_{\ell=0}^{mN-1}\frac{1}{R_\ell^p},
        \qquad
        \psi(r)=\frac{1}{r^2}-\frac{1}{r}.
\]

\begin{proposition}[Active-clock variance identity]\label{prop:active-clock}
For every $m,N,p$,
\begin{equation}\label{eq:active-clock}
        \Var_p(T_m(N))=
        \E_p\sum_{\ell=0}^{mN-1}\psi(R_\ell^p)
        +\Var_p(H_p).
\end{equation}
\end{proposition}

\begin{proof}
Conditional on the useful-hit path,
\[
        X_p=\sum_{\ell=0}^{mN-1}E_\ell,
        \qquad
        E_\ell\sim \Exp(R_\ell^p)
\]
independently.  Hence
\[
        \E[X_p\mid \text{path}]=H_p,
        \qquad
        \Var(X_p\mid \text{path})=\sum_{\ell=0}^{mN-1}\frac{1}{(R_\ell^p)^2}.
\]
The law of total variance gives
\[
        \Var(X_p)=\E\sum_\ell \frac{1}{(R_\ell^p)^2}+\Var(H_p).
\]
Since $\E X_p=\E H_p$ and $\Var(T_m(N))=\Var(X_p)-\E X_p$, \eqref{eq:active-clock} follows.
\end{proof}

\subsection{A direct global proof for \texorpdfstring{$m=1$}{m=1}}

The $m=1$ case was proved by Yu~\cite{Yu2017}.  The following proof is subsumed by the general theorem below, but it is included because it gives a useful active-clock mechanism.

\begin{theorem}[Conjecture~\ref{conj:DP2} for $m=1$]\label{thm:m1-global}
For every $N\ge2$ and every probability vector $p$,
\[
        \Var_p(T_1(N))\ge \Var_u(T_1(N)),
\]
with equality iff $p=u$.
\end{theorem}

\begin{proof}
For $m=1$, when $r$ coupons remain unseen, let $R_r$ be their total probability mass.  Then the uniform vector has $R_r^u=r/N$ deterministically.

Given an unseen set $S$ of size $r$ and total mass $R$, the next new coupon is $i\in S$ with probability $p_i/R$.  The new unseen mass is $R-p_i$, so
\[
        \E[R_{r-1}\mid S]=R-\frac{\sum_{i\in S}p_i^2}{R}.
\]
By Cauchy's inequality,
\[
        \sum_{i\in S}p_i^2\ge \frac{R^2}{r}.
\]
Therefore
\[
        \E[R_{r-1}\mid S]\le \left(1-\frac{1}{r}\right)R.
\]
Iterating from $R_N=1$ gives
\[
        \E R_r\le \frac{r}{N}.
\]
The function $\psi(r)=r^{-2}-r^{-1}$ is decreasing and convex on $(0,1]$.  Hence
\[
        \E\psi(R_r)\ge \psi(\E R_r)\ge \psi(r/N).
\]
By Proposition~\ref{prop:active-clock},
\[
        \Var_p(T_1(N))\ge \sum_{r=1}^N\E\psi(R_r)
        \ge \sum_{r=1}^N\psi(r/N).
\]
For the uniform vector, $H_u=\sum_r N/r$ is deterministic, so Proposition~\ref{prop:active-clock} gives
\[
        \Var_u(T_1(N))=\sum_{r=1}^N\psi(r/N).
\]
This proves the inequality.  If $p\ne u$, strictness occurs already in the Cauchy inequality at the first step.
\end{proof}

\begin{corollary}
For the ordinary equal-probability coupon collector,
\[
        \Var_u(T_1(N))=N^2H_N^{(2)}-NH_N.
\]
\end{corollary}

\subsection{Gamma-shape variance monotonicity}

The next lemma is the reusable module behind the full proof.  It says that for a Gamma-scaled random variable, a distributional derivative which is later than the size-biased current law forces the Poisson-corrected variance to increase.

\begin{lemma}[Gamma-shape variance monotonicity]\label{lem:shape-monotonicity}
Let $X_\theta\ge0$ have CDF $G_\theta$ and density $g_\theta$.  Put
\[
        W(\theta)=\Var(X_\theta)-\E X_\theta,
        \qquad
        \mu_\theta=\E X_\theta.
\]
Assume that $\Var(X_\theta)\ge \E X_\theta$ and that
\[
        w_\theta(x):=-\partial_\theta G_\theta(x)
\]
is nonnegative, locally integrable, and satisfies
$0<\int_0^\infty w_\theta(x)\dd x<\infty$.  Assume also that differentiation under the tail integrals below is justified.  If
\begin{equation}\label{eq:sizebiased-MLR}
        \frac{w_\theta(x)}{xg_\theta(x)}
\end{equation}
is increasing in $x$, then $W'(\theta)>0$.
\end{lemma}

\begin{proof}
The tail-integral identities give
\[
        \mu_\theta'=\int_0^\infty w_\theta(x)\dd x,
        \qquad
        \frac{d}{d\theta}\E X_\theta^2=2\int_0^\infty xw_\theta(x)\dd x.
\]
Therefore
\begin{equation}\label{eq:Wprime-general}
        W'(\theta)=\int_0^\infty [2x-(1+2\mu_\theta)]w_\theta(x)\dd x.
\end{equation}
The monotonicity hypothesis has the following standard likelihood-ratio interpretation.  Normalize
\[
        \nu_\theta(\dd x)=\frac{w_\theta(x)\dd x}{\int_0^\infty w_\theta(y)\dd y}.
\]
The size-biased law of $X_\theta$ has density proportional to $xg_\theta(x)$; hence \eqref{eq:sizebiased-MLR} says precisely that the Radon--Nikodym derivative of $\nu_\theta$ with respect to this size-biased law is increasing in $x$.  Thus $\nu_\theta$ dominates the size-biased law in monotone likelihood-ratio order.  Equivalently, by Chebyshev's integral inequality, its centroid is at least
\[
        \frac{\E X_\theta^2}{\E X_\theta}
        =\mu_\theta+\frac{\Var(X_\theta)}{\mu_\theta}.
\]
By the assumed inequality $\Var(X_\theta)\ge\mu_\theta$, this centroid is at least $\mu_\theta+1$.  Substituting this into \eqref{eq:Wprime-general} gives
\[
        W'(\theta)\ge \int_0^\infty w_\theta(x)\dd x>0.
\]
\end{proof}

\begin{remark}
In the coupon-collector application, $\Var(X_\theta)\ge\E X_\theta$ is automatic from the Poissonization identity
$\Var(T)=\Var(X)-\E X\ge0$.
\end{remark}

\subsection{The reverse-hazard ratio lemma}

Let
\[
        F(y)=F_m(y)=\Pbb(\GammaDist(m,1)\le y)
        =1-e^{-y}\sum_{k=0}^{m-1}\frac{y^k}{k!},
\]
and set
\[
        \phi(y)=\frac{F'(y)}{F(y)}.
\]
We need a one-site monotonicity property of this reverse hazard.

\begin{lemma}[Log-scale concavity of the reverse hazard]\label{lem:reverse-hazard-ratio}
For every $m\ge1$, the function
\[
        e(y):=\frac{d}{d\log y}\log\phi(y)
\]
is strictly negative and strictly decreasing on $(0,\infty)$.  Consequently, $\phi$ is strictly decreasing, and for every $c>1$ the ratio
\[
        y\mapsto \frac{\phi(cy)}{\phi(y)}
\]
is strictly decreasing.
\end{lemma}

\begin{proof}
Write
\[
        D_m(y)=e^y-\sum_{k=0}^{m-1}\frac{y^k}{k!}
        =\sum_{n=m}^{\infty}\frac{y^n}{n!}.
\]
Since
\[
        F(y)=e^{-y}D_m(y),
        \qquad
        F'(y)=e^{-y}\frac{y^{m-1}}{(m-1)!},
\]
we have
\[
        \phi(y)=\frac{y^{m-1}}{(m-1)!D_m(y)}.
\]
Therefore
\[
        e(y)=m-1-\frac{yD_m'(y)}{D_m(y)}.
\]
The quantity
\[
        \frac{yD_m'(y)}{D_m(y)}
\]
is the mean of the integer-valued random variable $K_y$ on $\{m,m+1,\ldots\}$ with weights proportional to $y^n/(n!)$.  In particular $\E K_y>m-1$, so $e(y)<0$ and $\phi'(y)=\phi(y)e(y)/y<0$.  Differentiating with respect to $\log y$ gives
\[
        \frac{d}{d\log y}\E K_y=\Var(K_y)>0.
\]
Thus $e(y)$ is strictly decreasing.  Finally,
\[
        \frac{d}{d\log y}\log\frac{\phi(cy)}{\phi(y)}=e(cy)-e(y)<0
\]
for $c>1$, proving the ratio statement.
\end{proof}

\subsection{Proof of finite variance extremality}

\begin{theorem}[Finite variance extremality]\label{thm:DP2-proof}
For every $m\ge1$, $N\ge2$, and every positive probability vector $p=(p_1,\ldots,p_N)$,
\[
        \Var_p(T_m(N))\ge \Var_u(T_m(N)),
        \qquad
        u=(1/N,\ldots,1/N),
\]
with equality iff $p=u$.
\end{theorem}

\begin{proof}
It is enough to prove strict radial monotonicity away from $u$.  Fix a nonuniform vector $p$, and set
\[
        p_i(\theta)=\frac{1}{N}+\theta h_i,
        \qquad
        h_i=p_i-\frac{1}{N},
        \qquad
        0\le \theta\le1.
\]
For $\theta>0$, put $q_i=p_i(\theta)$.  The Poissonized completion time has CDF
\[
        G_\theta(t)=\prod_{i=1}^NF(q_it).
\]
Its density is
\[
        g_\theta(t)=G_\theta(t)\sum_{i=1}^Nq_i\phi(q_it).
\]
Let
\[
        w_\theta(t):=-\partial_\theta G_\theta(t).
\]
Since
\[
        \partial_\theta\log G_\theta(t)=t\sum_i h_i\phi(q_it),
\]
and $h_i=(q_i-1/N)/\theta$, we obtain
\[
        w_\theta(t)=\frac{G_\theta(t)t}{\theta}
        \left[\frac{1}{N}\sum_i\phi(q_it)-\sum_iq_i\phi(q_it)\right].
\]
The bracket is nonnegative.  Indeed, $q_i$ and $\phi(q_it)$ are oppositely ordered because $\phi$ is decreasing; hence the average of the $q_i$ with weights proportional to $\phi(q_it)$ is at most the unweighted average $1/N$.  Since $p$ is nonuniform, the bracket is strictly positive for $t>0$.

Define
\[
        C(t)=\sum_i\phi(q_it),
        \qquad
        B(t)=\sum_iq_i\phi(q_it),
\]
and
\[
        M(t)=\frac{B(t)}{C(t)}.
\]
Then
\begin{equation}\label{eq:w-over-tg}
        \frac{w_\theta(t)}{t g_\theta(t)}
        =\frac{1}{\theta}\left(\frac{C(t)}{NB(t)}-1\right)
        =\frac{1}{\theta}\left(\frac{1}{NM(t)}-1\right).
\end{equation}
We now show that $M(t)$ is decreasing.  Let
\[
        \alpha_i(t)=\frac{\phi(q_it)}{C(t)}.
\]
Then
\[
        M(t)=\sum_iq_i\alpha_i(t).
\]
With $e(y)=d\log\phi(y)/d\log y$,
\[
        \frac{dM}{d\log t}
        =\sum_i q_i\alpha_i(t)\left(e(q_it)-\sum_j\alpha_j(t)e(q_jt)\right)
        =\Cov_{\alpha(t)}(q_i,e(q_it)).
\]
By Lemma~\ref{lem:reverse-hazard-ratio}, $e$ is decreasing.  Thus $q_i$ and $e(q_it)$ are oppositely ordered, and the covariance is nonpositive; it is strictly negative unless all $q_i$ are equal.  Hence $M(t)$ is strictly decreasing.  From \eqref{eq:w-over-tg}, $w_\theta(t)/(tg_\theta(t))$ is strictly increasing.

The remaining hypotheses of Lemma~\ref{lem:shape-monotonicity} are satisfied here.  The function $G_\theta(t)=\prod_iF(q_it)$ is smooth in $(\theta,t)$ for $\theta>0$ and $t>0$; its Gamma tails justify differentiating under the first two tail integrals; and the formula above gives $w_\theta(t)\ge0$ with finite positive integral because $p$ is nonuniform.  Finally, $\Var(X_{p(\theta)})\ge\E X_{p(\theta)}$ follows from the Poissonization identity $\Var(T)=\Var(X)-\E X\ge0$.

By Lemma~\ref{lem:shape-monotonicity},
\[
        \frac{d}{d\theta}\left(\Var(X_{p(\theta)})-\E X_{p(\theta)}\right)>0.
\]
But
\[
        \Var_{p(\theta)}(T_m(N))=\Var(X_{p(\theta)})-\E X_{p(\theta)}.
\]
Therefore $\theta\mapsto \Var_{p(\theta)}(T_m(N))$ is strictly increasing on $(0,1]$.  By continuity at $\theta=0$, integrating this strict derivative over $(0,1]$ gives $\Var_p(T_m(N))>\Var_u(T_m(N))$ for every nonuniform $p$.  This proves the claim.
\end{proof}

\begin{corollary}[Strict local minimum at uniform]\label{cor:strict-local-main}
For every $m\ge1$ and $N\ge2$, the uniform vector $u$ is a strict local minimizer of $p\mapsto\Var_p(T_m(N))$ on the probability simplex.
\end{corollary}

\begin{proof}
This is immediate from Theorem~\ref{thm:DP2-proof}.  More precisely, for every nonzero tangent direction $h$ with $\sum_i h_i=0$, the curve $u+\varepsilon h$ is nonuniform for every sufficiently small nonzero $\varepsilon$, and Theorem~\ref{thm:DP2-proof} gives
\[
        \Var_{u+\varepsilon h}(T_m(N))>\Var_u(T_m(N)).
\]
An independent Hessian proof, useful for comparison with other terminal-defect models, is given in Appendix~\ref{app:hessian}.
\end{proof}

\subsection{Why simpler local principles are insufficient}

Several tempting stronger principles are false, even though Conjecture~\ref{conj:DP2} is true.
\begin{itemize}[leftmargin=*]
\item Pairwise variance smoothing need not hold: averaging two coordinates may lower the mean while increasing the variance.  This already occurs for $m=1$, so the variance is not simply Schur-convex by elementary smoothing.
\item Pointwise Bellman comparison of deficit states is false.  A nonuniform vector can be locally less variable at states where the active coupons happen to be high-probability coupons.
\item Pairwise frontier terms in the radial profile decomposition need not be individually late-tail enough.  The proof works because the full radial derivative measure is size-biased later than the current law, an aggregate property not visible term by term.
\end{itemize}

\section{Concluding remarks and open problems}

The terminal-defect viewpoint gives a unified and reusable proof structure for the double Dixie cup problem.  The equal-probability growing-$m$ theorem and moment module recover the fixed-$m$ equal-probability variance asymptotic stated by Doumas--Papanicolaou~\cite{DoumasPapanicolaou2016}; the case $m=1$ is classical by Brayton~\cite{Brayton1963}, while the present module treats fixed $m\ge2$ uniformly and extends to arbitrary growing multiplicity.  The unequal-probability modules give direct terminal-defect versions of the two main asymptotic regimes in~\cite{DoumasPapanicolaou2016}: infinite products in Case I and endpoint Laplace estimates in Case II\@.  The power-law theorem verifies the endpoint module for the concrete family $p_{N,j}\propto j^{-\alpha}$.

The finite variance-minimization conjecture is also resolved here.  The proof does not rely on pairwise smoothing or statewise Bellman comparison, both of which are false in stronger forms.  Instead it uses a global radial comparison: the derivative of the Poissonized completion distribution away from uniform is later than the size-biased current law.  The one-site reason is the log-scale monotonicity of the Gamma CDF reverse hazard
\[
        \phi_m(y)=\frac{F_m'(y)}{F_m(y)}.
\]
This gives a compact proof of the finite extremality statement and explains why local smoothing arguments can fail while the aggregate radial inequality remains true.

Several directions remain open or worth developing further.
\begin{enumerate}[label=\textup{(\arabic*)},leftmargin=*]
\item Verify the endpoint-Laplace hypothesis for broader differentiability classes beyond the power-law family treated here.
\item Extend the radial size-biased MLR module to related terminal-defect collectors, such as grouped coupons, structured spatial collectors, and nonmonotone coupon processes.
\item Investigate whether the active-clock identity admits analogous global radial monotonicity proofs in models where no independent Gamma-shape representation is available.
\item Develop quantitative stability: estimate the variance gap $\Var_p(T_m(N))-\Var_u(T_m(N))$ in terms of distances such as $\sum_i(p_i-1/N)^2$ or $D_{\mathrm{KL}}(p\Vert u)$.
\end{enumerate}

\appendix

\section{Gamma-tail quantile estimates}\label{app:gamma-tail}

This appendix records the gamma-tail estimates used in Lemma~\ref{lem:gamma-tail-inversion}.  The point is not to reprove Temme's saddle-point expansion, but to isolate exactly which consequences of the uniform incomplete-gamma asymptotics are needed in the probabilistic argument.

Throughout this appendix
\[
        Q_m(t)=e^{-t}\sum_{j=0}^{m-1}{\frac{t^j}{j!}},
        \qquad
        f_m(t)=e^{-t}{\frac{t^{m-1}}{(m-1)!}},
        \qquad
        h_m(t)={\frac{f_m(t)}{Q_m(t)}}.
\]
Thus $Q_m(t)=Q(m,t)$ in Temme's notation, where
$Q(a,x)=\Gamma(a,x)/\Gamma(a)$ is the normalized upper incomplete gamma function.  Temme's turning-point variable is
\[
        {\frac12}\eta^2=\lambda-1-\log\lambda,
        \qquad
        \lambda={\frac{x}{a}},
        \qquad
        \operatorname{sign}\eta=\operatorname{sign}(\lambda-1).
\]
Temme~\cite[Sec.~1, (1.1), (1.3)--(1.6)]{Temme1979} proves the representation
\[
        Q(a,x)={\frac12}\operatorname{erfc}(\eta\sqrt{a/2})+R_a(\eta)
\]
with an expansion for $R_a(\eta)$ that is uniform for real $\eta$, equivalently uniform for $x\in[0,\infty)$; the differentiated expansion and remainder representation are given in~\cite[Sec.~2, (2.1)--(2.14)]{Temme1979}.  Temme's inversion paper restates this uniform expansion in~\cite[Sec.~2, (2.1)--(2.7)]{Temme1992} and treats the inverse problem $Q(a,x)=q$ by first solving
\[
        {\frac12}\operatorname{erfc}(\eta_0\sqrt{a/2})=q
\]
and then writing $\eta=\eta_0+\varepsilon(\eta_0,a)$ with an expansion in powers of $a^{-1}$; see~\cite[Sec.~3, (3.1)--(3.10)]{Temme1992}.  The analytic and end-tail behavior of the inversion coefficients is recorded in~\cite[Sec.~5, (5.1)--(5.3)]{Temme1992}.  In the proof below, however, we do not need to assume a uniform inversion remainder in the moving extreme-tail range $q=1/n$.  We use Temme's expansion only to identify the correct signed turning-point scale near $x=a$; the high-quantile upper bound is obtained from Chernoff's inequality, and the local hazard estimates are elementary consequences of the finite-sum formula for integer shapes.

\begin{proposition}[Gamma quantile estimates]\label{prop:gamma-quantile-estimates}
Let $m_n\ge1$, let $b_n$ be the unique solution of
\[
        Q_{m_n}(b_n)=1/n,
\]
and put
\[
        a_n={\frac{Q_{m_n}(b_n)}{f_{m_n}(b_n)}}={\frac1{h_{m_n}(b_n)}}.
\]
Then, uniformly for $x$ in compact subsets of $\mathbb R$,
\[
        {\frac{Q_{m_n}(b_n+a_nx)}{Q_{m_n}(b_n)}}\to e^{-x}.
\]
Moreover,
\[
        {\frac{Q_{m_n}(b_n+a_nx)}{Q_{m_n}(b_n)}}\le e^{-x}
        \qquad (x\ge0),
\]
there are constants $c>0$ and $R_n\to\infty$ such that
\[
        nQ_{m_n}(b_n-a_nx)\ge c e^{cx},
        \qquad 0\le x\le R_n,
\]
the window may be chosen so slowly that
\[
        R_n\le {\frac{b_n}{2a_n}},
        \qquad
        \left(1+{\frac{b_n}{a_n}}\right)^r\exp\{-c e^{cR_n}\}\to0
\]
for every fixed $r>0$, and
\[
        {\frac{n a_n^2}{b_n}}\to\infty.
\]
\end{proposition}

\begin{proof}
The existence and uniqueness of $b_n$ are elementary, since $Q_m(0)=1$, $Q_m(t)\to0$ as $t\to\infty$, and $Q_m'(t)=-f_m(t)<0$ for $t>0$.

First suppose that the sequence $m_n$ is bounded.  Passing to finitely many fixed shapes, we may use the elementary expansion
\[
        Q_m(y)\sim e^{-y}{\frac{y^{m-1}}{(m-1)!}},
        \qquad
        h_m(y)={\frac{f_m(y)}{Q_m(y)}}\to1,
        \qquad y\to\infty,
\]
uniformly over the finitely many possible values of $m$.  Hence $b_n=\log n+O(\log\log n)$, $a_n\to1$, and
\[
        \sup_{|u|\le A}|a_nh_{m_n}(b_n+a_nu)-1|\to0
\]
for each fixed $A<\infty$.  The local ratio estimate follows by integrating $d\log Q_m(t)/dt=-h_m(t)$.  The left-window estimate follows on any window $R_n\to\infty$ with $R_n=o(b_n)$ and sufficiently slow growth; thinning $R_n$ if necessary gives the displayed window condition.  Finally,
\[
        {\frac{n a_n^2}{b_n}}\sim {\frac{n}{\log n}}\to\infty.
\]

It remains to consider the case $m_n\to\infty$.  Write
\[
        \alpha=m_n,\qquad k=\alpha-1,\qquad b=b_n,\qquad
        \lambda={\frac b\alpha}.
\]
Let $\eta_n$ be Temme's signed turning-point parameter associated with $(\alpha,b)$, and put
\[
        z_n=\eta_n\sqrt{\alpha}.
\]
We first locate the quantile on this scale.  At $x=\alpha$ one has $\eta=0$, and Temme's uniform expansion gives $Q(\alpha,\alpha)=1/2+O(\alpha^{-1/2})$.  Since $Q_\alpha(b)=1/n\to0$, it follows that $b>\alpha$ for all sufficiently large $n$, so $\eta_n>0$.  If $z_n$ failed to tend to infinity, then along a subsequence $0\le z_n\le A$ for some fixed $A$.  On that subsequence $\eta_n=O(\alpha^{-1/2})$, and Temme's expansion at the turning point gives
\[
        Q_\alpha(b)={\frac12}\operatorname{erfc}(z_n/\sqrt2)+O(\alpha^{-1/2}),
\]
which is bounded below by a positive constant depending only on $A$.  This contradicts $Q_\alpha(b)=1/n\to0$.  Hence $z_n\to\infty$.

For the upper bound on $z_n$, use the Chernoff bound for a Gamma$(\alpha,1)$ variable $S_\alpha$: for $\lambda\ge1$,
\[
        \Pbb(S_\alpha\ge \alpha\lambda)
        \le \exp\{-\alpha(\lambda-1-\log\lambda)\}
        =\exp\{-z_n^2/2\}.
\]
Because $Q_\alpha(b)=1/n$, this gives
\begin{equation}\label{eq:temme-quantile-order}
        z_n\to\infty,
        \qquad
        z_n^2\le 2\log n.
\end{equation}
This is the only place where Temme's uniform expansion is used in the large-shape proof.

For $\lambda\ge1$ the quantities $\eta^2=2(\lambda-1-\log\lambda)$ and $(\lambda-1)^2/\lambda$ are comparable up to absolute constants.  Therefore, with
\[
        \delta_n=1-{\frac{k}{b}},
\]
we have
\begin{equation}\label{eq:delta-scale}
        b\delta_n^2\to\infty,
        \qquad
        b\delta_n^2=O(1+\log n).
\end{equation}
Indeed, replacing $k=\alpha-1$ by $\alpha$ changes $b\delta_n^2$ only by lower-order terms, while
\[
        \alpha\,{\frac{(\lambda-1)^2}{\lambda}}
\]
is comparable to $z_n^2=\alpha\eta_n^2$.

We now pass from the turning-point scale to the hazard scale using only the integer-shape identity
\begin{equation}\label{eq:R-integral}
        R_k(t):={\frac{Q_{k+1}(t)}{f_{k+1}(t)}}
        =\int_0^\infty e^{-s}\left(1+{\frac{s}{t}}\right)^k\dd s,
        \qquad t>0.
\end{equation}
For $t>k$ set $\delta(t)=1-k/t$ and $T(t)=t\delta(t)^2$.  We claim that if $T(t)\to\infty$, then
\begin{equation}\label{eq:R-asymp}
        \delta(t)R_k(t)\to1,
        \qquad
        R_k'(t)\to0.
\end{equation}
The upper bound follows immediately from $\log(1+s/t)\le s/t$:
\[
        R_k(t)\le\int_0^\infty e^{-\delta(t)s}\dd s={\frac1{\delta(t)}}.
\]
For the matching lower bound, put $L=T(t)^{1/4}/\delta(t)$.  Then $\delta(t)L\to\infty$ and $L/t\to0$.  Uniformly for $0\le s\le L$,
\[
        k\log\left(1+{\frac{s}{t}}\right)
        \ge {\frac{ks}{t}}-O\left({\frac{s^2}{t}}\right),
\]
and $s^2/t\le L^2/t=T(t)^{-1/2}=o(1)$.  Hence
\[
        R_k(t)\ge(1-o(1))\int_0^L e^{-\delta(t)s}\dd s
        =(1-o(1)){\frac1{\delta(t)}}.
\]
This proves $\delta(t)R_k(t)\to1$.  Differentiating \eqref{eq:R-integral}, or equivalently differentiating $Q/f$, gives
\begin{equation}\label{eq:R-derivative}
        R_k'(t)=\delta(t)R_k(t)-1,
\end{equation}
so $R_k'(t)\to0$ as well.

At the quantile $b=b_n$, \eqref{eq:R-asymp} and \eqref{eq:delta-scale} give
\begin{equation}\label{eq:hazard-clock-final}
        b_n h_{m_n}(b_n)^2\sim b_n\delta_n^2=O(1+\log n),
        \qquad
        b_nh_{m_n}(b_n)\sim b_n\delta_n\to\infty.
\end{equation}
Since $a_n=1/h_{m_n}(b_n)$, this proves
\[
        {\frac{n a_n^2}{b_n}}
        ={\frac{n}{b_nh_{m_n}(b_n)^2}}
        \ge {\frac{c n}{1+\log n}}\to\infty
\]
and also $b_n/a_n=b_nh_{m_n}(b_n)\to\infty$.

We next prove the local hazard estimate.  Fix $A<\infty$ and let
\[
        t=b_n+a_nu,
        \qquad |u|\le A.
\]
Because $a_n\sim1/\delta_n$ and $b_n\delta_n\to\infty$, we have $t/b_n\to1$ and $\delta(t)/\delta_n\to1$ uniformly for $|u|\le A$.  Hence $T(t)=t\delta(t)^2\to\infty$ uniformly on this bounded $u$-window.  By \eqref{eq:R-asymp}--\eqref{eq:R-derivative}, $R_k'(t)=o(1)$ uniformly there.  Since the length of the $t$-window is $O(a_n)=O(R_k(b_n))$, it follows that
\[
        {\frac{R_k(b_n+a_nu)}{R_k(b_n)}}\to1
        \qquad (|u|\le A),
\]
uniformly.  Therefore
\begin{equation}\label{eq:app-local-hazard-final}
        \sup_{|u|\le A}\left|a_nh_{m_n}(b_n+a_nu)-1\right|\to0.
\end{equation}
Integrating the logarithmic derivative gives, uniformly for bounded $x$,
\[
\log{\frac{Q_{m_n}(b_n+a_nx)}{Q_{m_n}(b_n)}}
        =-\int_0^x a_nh_{m_n}(b_n+a_nu)\dd u\to -x,
\]
and hence the local ratio limit.

For $x\ge0$, the Gamma upper-tail hazard is increasing.  Since $a_n=1/h_{m_n}(b_n)$,
\[
        \log{\frac{Q_{m_n}(b_n+a_nx)}{Q_{m_n}(b_n)}}
        =-\int_0^x a_nh_{m_n}(b_n+a_nu)\dd u\le -x,
\]
which proves the one-sided right-tail bound.

Finally, the local convergence in \eqref{eq:app-local-hazard-final} may be made uniform on a slowly growing window by a diagonal choice.  Thus there is $\widetilde R_n\to\infty$ such that
\[
        {\frac12}\le a_nh_{m_n}(b_n+a_nu)\le {\frac32},
        \qquad |u|\le \widetilde R_n.
\]
Choose $R_n\to\infty$ with
\[
        R_n\le \widetilde R_n,
        \qquad
        R_n\le {\frac{b_n}{2a_n}},
        \qquad
        \left(1+{\frac{b_n}{a_n}}\right)^r\exp\{-\tfrac12 e^{R_n/2}\}\to0
\]
for each fixed $r>0$; this is possible because $b_n/a_n\to\infty$ and $\widetilde R_n\to\infty$, and the double exponential dominates any prescribed polynomial rate after thinning the window.  For $0\le x\le R_n$,
\[
        nQ_{m_n}(b_n-a_nx)
        ={\frac{Q_{m_n}(b_n-a_nx)}{Q_{m_n}(b_n)}}
        =\exp\left\{\int_0^x a_nh_{m_n}(b_n-a_nu)\dd u\right\}
        \ge e^{x/2}.
\]
This is the asserted left-tail estimate, after reducing the absolute constant $c$ if necessary.  The proposition follows.
\end{proof}

\section{Independent local Hessian proof}\label{app:hessian}

Although Corollary~\ref{cor:strict-local-main} now follows from the global extremality theorem, the following alternate argument is useful because it exposes a second positivity mechanism: infinitesimal heterogeneity creates a strengthened size-bias of the uniform Gamma maximum.

\begin{theorem}[Strict local Hessian positivity]\label{thm:local-all-m}
Fix $m\ge1$ and $N\ge2$.  Let
\[
        V_{m,N}(p)=\Var_p(T_m(N)).
\]
If $\sum_i h_i=0$, then
\[
        D^2V_{m,N}(u)[h,h]=C_{m,N}\sum_{i=1}^Nh_i^2
\]
with $C_{m,N}>0$.
\end{theorem}

\begin{proof}
By symmetry it is enough to compute the Hessian in the two-marked direction
\[
        p_1={\frac{1}{N}}+\varepsilon,
        \qquad
        p_2={\frac{1}{N}}-\varepsilon,
        \qquad
        p_3=\cdots=p_N={\frac{1}{N}}.
\]
Let $F=F_m$ and let $E_u=\E_uT_m(N)$.  Differentiating the Poissonized profile twice and using the exact variance identity gives the tangent eigenvalue
\begin{equation}\label{eq:Cmn-integral-new}
        C_{m,N}=N^3\int_0^\infty y^2F(y)^N[-(\log F)''(y)]
        [2Ny-(1+2E_u)]\,dy.
\end{equation}
Let $Y_N=\max_{1\le i\le N}G_i$, where the $G_i$ are iid $\GammaDist(m,1)$, and define
\[
        a_m(y)=y^2\left({\frac{F'(y)}{F(y)}}-{\frac{F''(y)}{F'(y)}}\right).
\]
Since the density of $Y_N$ is $NF(y)^{N-1}F'(y)$ and $E_u=N\E Y_N$, \eqref{eq:Cmn-integral-new} is equivalent to
\begin{equation}\label{eq:Cmn-cov-app}
        {\frac{C_{m,N}}{N^2}}=2N\Cov(Y_N,a_m(Y_N))-\E a_m(Y_N).
\end{equation}
Now
\[
        {\frac{F''(y)}{F'(y)}}={\frac{m-1}{y}}-1,
\]
so
\[
        a_m(y)=y\,b_m(y),
        \qquad
        b_m(y)=y{\frac{F'(y)}{F(y)}}+y-(m-1).
\]
We now prove directly that $b_m$ is strictly increasing.  Put
\[
        D_m(y)=e^y-\sum_{k=0}^{m-1}{\frac{y^k}{k!}}
        =\sum_{n=m}^{\infty}{\frac{y^n}{n!}}.
\]
Since $F_m(y)=e^{-y}D_m(y)$ and $F_m'(y)=e^{-y}y^{m-1}/\Gamma(m)$, we have
\[
        b_m(y)={\frac{y^m}{(m-1)!D_m(y)}}+y-(m-1).
\]
Differentiating gives
\[
        b_m'(y)=
        {\frac{(m-1)!D_m(y)^2+m y^{m-1}D_m(y)-y^mD_m'(y)}{(m-1)!D_m(y)^2}}.
\]
It remains to show that the numerator is positive.  The coefficient of $y^{2m+s}$, $s\ge0$, in the numerator equals
\[
        (m-1)!\sum_{r=0}^{s}{\frac{1}{(m+r)!(m+s-r)!}}
        -{\frac{s+1}{(m+s+1)!}}.
\]
After multiplication by $\Gamma(m+s+2)$, this coefficient becomes
\[
        \sum_{r=0}^{s}
        {\frac{(m-1)!(m+s+1)!}{(m+r)!(m+s-r)!}}-(s+1).
\]
For each $0\le r\le s$,
\[
        {\frac{(m-1)!(m+s+1)!}{(m+r)!(m+s-r)!}}
        =\prod_{j=0}^{r}{\frac{m+s-r+1+j}{m+j}}>1.
\]
Hence every coefficient is strictly positive, and therefore $b_m'(y)>0$ for $y>0$.  Also $D_m(y)\sim y^m/\Gamma(m+1)$ as $y\downarrow0$, so $b_m(0+)=1$.  Thus $b_m(y)>1$ for $y>0$, and in particular $a_m(y)=yb_m(y)>0$.

Consequently, for any positive random variable $Y$,
\[
        {\frac{\Cov(Y,a_m(Y))}{\E a_m(Y)}}\ge {\frac{\Var(Y)}{\E Y}}.
\]
Indeed, if $Y'$ is an independent copy of $Y$, then
\[
\E[Y^2b_m(Y)]\E[Y]-\E[Yb_m(Y)]\E[Y^2]
={\frac{1}{2}}\E\{YY'(Y-Y')(b_m(Y)-b_m(Y'))\}\ge0,
\]
because $b_m$ is increasing.
Applying this to $Y=Y_N$ gives
\[
        {\frac{C_{m,N}}{N^2}}\ge \E a_m(Y_N)\left({\frac{2N\Var(Y_N)}{\E Y_N}}-1\right).
\]
It remains to show the bracket is positive.  Write $G_i=SP_i$, where
\[
        S\sim\GammaDist(mN,1),
        \qquad
        P=(P_1,\ldots,P_N)\sim\operatorname{Dirichlet}(m,\ldots,m),
        \qquad
        S\perp P.
\]
If $M=\max_iP_i$, then $Y_N=SM$.  Therefore
\[
        \E Y_N=mN\E M,
\]
and
\[
        \Var(Y_N)=mN\E M^2+(mN)^2\Var(M).
\]
Since $M\ge1/N$ almost surely and $M>1/N$ with positive probability,
\[
        2N\Var(Y_N)\ge 2mN^2\E M^2>mN\E M=\E Y_N.
\]
Thus $C_{m,N}>0$.
\end{proof}

\end{document}